\documentclass[letterpaper,11pt]{amsart}
\usepackage[margin=1in]{geometry}
\usepackage{amsfonts}
\usepackage{pb-diagram}
\title[Teichm\"uller and Alexander polynomials]{Teichm\"uller polynomials, Alexander polynomials and finite covers of surfaces}

\newtheorem{thm}{Theorem}
\newtheorem{lemma}[thm]{Lemma}

\newtheorem{cor}[thm]{Corollary}
\newtheorem{prop}[thm]{Proposition}

\numberwithin{thm}{section}

\newcommand{\eps}{\epsilon}

\newcommand{\bZ}{\mathbb{Z}}
\newcommand{\bR}{\mathbb{R}}
\newcommand{\bC}{\mathbb{C}}

\DeclareMathOperator{\Hom}{Hom}

\DeclareMathOperator{\Mod}{Mod}

\DeclareMathOperator{\rk}{rk}

\newcommand{\bQ}{\mathbb{Q}}

\newcommand{\bU}{\mathbb{U}}

\author[T. Koberda]{Thomas Koberda}
\address{Department of Mathematics\\ Harvard University\\ 1 Oxford St.\\ Cambridge, MA 02138 }
\email{ koberda@math.harvard.edu}
\subjclass[2010]{Primary: 57M10; Secondary: 57M27}
\keywords{Pseudo-Anosov homeomorphism, twisted homology, Teichm\"uller polynomial, Alexander polynomial, Thurston norm ball}

\begin{document}
\begin{abstract}
In this article we explore a connection between finite covers of surfaces and the Teichm\"uller polynomial of a fibered face of a hyperbolic $3$--manifold.  We consider the action of a homological pseudo-Anosov homeomorphism $\psi$ on the homology groups of a class of finite abelian covers of a surface $\Sigma_{g,n}$.  Eigenspaces of the deck group actions on these covers are naturally parametrized by rational points on a torus.  We show that away from the trivial eigenspace, the spectrum of the action of $\psi$ on these eigenspaces is bounded away from the dilatation of $\psi$.  We show that the action $\psi$ on these eigenspaces is governed by the Teichm\"uller polynomial.
\end{abstract}
\maketitle
\begin{center}
\today
\end{center}
\section{Introduction}
Let $M$ be a matrix over the ring of integral Laurent polynomials in some (finite) number of variables.  We say that $M$ is Perron--Frobenius if there is a power $M^k$ of $M$ such that every entry of $M^k$ is a nontrivial Laurent polynomial with positive coefficients.  In this note, we shall establish the following fact:

\begin{prop}\label{p:pf}
Let $M$ be a square Perron--Frobenius matrix over $\bZ[t,t^{-1}]$ and let $\phi\in\Hom(\bZ,S^1)$ be a non--torsion point.  Then the spectral radius of $\phi(M)$ is strictly smaller than that of $\phi_0(M)$, where $\phi_0$ is the trivial representation of $\bZ$.
\end{prop}

In general, we will write $H$ for a finitely generated, torsion--free abelian group, and $\bZ[H]$ is the ring of integral Laurent polynomials over $H$.
We remark that if $M$ is a Perron--Frobenius matrix over the ring of Laurent polynomials $\bZ[H]$ then applying the trivial representation of $H$ to $M$ gives us a usual integral Perron--Frobenius matrix, which in turn has a unique real eigenvalue of maximal modulus.

Let $M$ be a fibered hyperbolic manifold with first Betti number $h+1$, where $h=\rk H\geq 1$.  Choosing a fibration of $M$, we can decompose $H_1(M,\bZ)/(Torsion)$ as $\bZ\oplus H$.  The corresponding fibered face of $M$ has an associated {\bf Teichm\"uller polynomial} $\theta(u,t)\in\bZ[\bZ\oplus H]$, with $t\in H$.  Proposition \ref{p:pf} has the following corollary:

\begin{cor}\label{c:teich}
Let $K$ be the largest root of $\theta(u,\phi_0(t))$, and let $\phi\in\Hom(H,S^1)$ be a non--torsion point.  The largest root of the polynomial $\theta(u,\phi(t))$ has modulus strictly less than $K$.
\end{cor}

Since we are interested in the actions of pseudo-Anosov homeomorphisms on the homology of finite covers of the base surface, it would be useful to relate Corollary \ref{c:teich} to homology of finite covers of the base surface.

Let $\Sigma_{g,n}$ be a surface of genus $g$ with $n$ punctures and let $\psi\in\Mod_{g,n}$, the mapping class group of $\Sigma_{g,n}$.  Let $\Sigma_{h,k}\to\Sigma_{g,n}$ be a finite cover to which $\psi$ lifts.  Each lift of $\psi$ acts on the homology of $\Sigma_{h,k}$.  This action can be very complicated, especially as we allow the cover itself to vary (cf. \cite{KobGeom} or \cite{nilpcover}, for instance).  In this note, we will be interested in the spectrum of these actions as the cover is allowed to vary over a certain class of abelian covers.

Let $\psi$ be a pseudo-Anosov homeomorphism.  Recall that $\psi$ stabilizes a canonical quadratic differential $q\in \mathcal{Q}(\Sigma_{g,n})$.  The quadratic differential $q$ gives rise to two foliations which, away from the zeros of $q$, furnish local charts to $\bC$ together with a horizontal and vertical foliation.  The action of $\psi$ stretches the horizontal foliation by a factor of $K$ and contracts the vertical foliation by a factor of $K$, where $K$ is the {\bf dilatation} of $\psi$.

When $q=\omega^2$ for some $1$--form $\omega$, the mapping class $\psi$ is termed a {\bf homological pseudo-Anosov}.  In this case, the foliations of $q$ admit global orientations which are preserved by $\psi$.  It turns out in this case that $\omega$ gives rise an eigenclass for the action of $\psi$ on $H^1(\Sigma_{g,n},\bR)$ whose corresponding eigenvalue is also $K$.

It is known (cf. \cite{KS} and \cite{LT}) that if $\psi$ is a homological pseudo-Anosov homeomorphism then $K$ is a simple eigenvalue of largest modulus.  We will be interested in the size of the second largest eigenvalue of the action of $\psi$ on the homology of certain abelian covers of $\Sigma_{g,n}$.

The connection between Teichm\"uller polynomials and homology is achieved by appealing to the theory of the {\bf Alexander polynomial} $A(u,t)$.  The Alexander polynomial can be defined for any CW complex (see \cite{ctm4}), and in the case of a fibered hyperbolic $3$--manifold, the Alexander polynomial is very useful for describing the action of a mapping class on the homology of certain finite abelian covers of the base surface.  The following result is not particularly new, but will allow us to see the connection between Teichm\"uller polynomials and homology of finite covers more clearly:

\begin{prop}\label{p:alexteich}
Let $\Sigma\to M\to S^1$ be a fibered hyperbolic $3$--manifold with monodromy $\psi$ and let $\theta(u,t)$ and $A(u,t)$ be the associated Teichm\"uller and Alexander polynomials, viewed as polynomials in one variable over $\bZ[H]$.
\begin{enumerate}
\item
If the monodromy $\psi$ is a homological pseudo-Anosov homeomorphism then $A(u,t)$ divides $\theta(u,t)$.
\item
For each irreducible character $\chi$ of $H$, the polynomial $A(u,\chi(t))$ is the characteristic polynomial of the action of $\psi$ on the twisted homology group $H_1(\Sigma_{g,n},\bC_{\chi})$.
\end{enumerate}
\end{prop}

Suppose that the action of $\psi$ on $H_1(\Sigma_{g,n},\bZ)$ has a nonzero fixed vector.  The fixed vectors of the $\psi$--action give an infinite family of finite abelian covers to which $\psi$ lifts and commutes with the deck group.  We will write $H$ for the fixed subgroup of $H_1(\Sigma_{g,n},\bZ)$ and we will write $h$ for its rank.  If $\Sigma_{h,k}\to\Sigma_{g,n}$ is such a cover with deck group $A$ then the complex homology of $\Sigma_{h,k}$ splits into eigenspaces according to the representations of $A$.  Furthermore, $\psi$ acts on each of these eigenspaces.  As $\Sigma_{h,k}$ varies over all finite abelian covers given by $H$, we can parametrize the eigenspaces of the deck group by torsion points on the torus $\bU^h\cong\Hom(H,S^1)$, where $\bU\cong S^1$ denotes the unit complex numbers.  We will usually not need to choose a basis for $H$ so that the isomorphism $\bU^h\cong\Hom(H,S^1)$ is not canonical.

We thus obtain a continuous function $\rho$ on $\Hom(H,S^1)$ which assigns to a point the spectral radius of the action of $\psi$ on the corresponding eigenspace, which is to say the largest eigenvalue of the action of $\psi$ on the homology group $H_1(\Sigma_{g,n},\bC_{\phi})$.

We will establish the following result as a corollary of Corollary \ref{c:teich}:

\begin{thm}\label{t:main}
Let $N$ be an arbitrary open neighborhood of the identity in $\Hom(H,S^1)$.  Then there is a $\delta=\delta(N)>0$ such that $\rho(\phi)\leq K-\delta$ for each $\phi\in\Hom(H,S^1)\setminus N$.
\end{thm}

Note that the continuity of $\rho$ implies that $\rho(\phi)$ will tend to $K$ as $\phi$ tends to the identity.  This last observation should be contrasted with McMullen's result in \cite{ctm1}.  There, he proves that if $\psi$ is a non--homological pseudo-Anosov which remains non--homological on every finite cover of $\Sigma_{g,n}$, then the largest eigenvalue of the action of $\psi$ on $H_1(\Sigma_{h,k},\bC)$ is bounded away from $K$ as $\Sigma_{h,k}$ varies over all finite covers of $\Sigma_{g,n}$.

To see the contrast more clearly, suppose that $\psi$ is a non--homological pseudo-Anosov homeomorphism with nonzero invariant cohomology on a surface $\Sigma$, and suppose that $\psi$ lifts to a homological pseudo-Anosov homeomorphism (which we also call $\psi$) on a double branched cover $\Sigma'$ of $\Sigma$.  Theorem \ref{t:main} implies that there are $\psi$--modules occurring in the homology of finite abelian covers of $\Sigma'$ which are never found as submodules of the homology of any finite cover of $\Sigma$.

The continuity of the function $\rho$ in Theorem \ref{t:main}, combined with Proposition \ref{p:alexteich} has the following immediate consequence.  Let $\psi\in\Mod_{g,n}$ be a homological pseudo-Anosov with nonzero invariant cohomology and let $\Sigma_{h,k}\to\Sigma_{g,n}$ be an abelian cover to which $\psi$ lifts.  Write $\gamma$ for the absolute value of the difference between the absolute values of the top eigenvalue (the dilatation) of $\psi$ on $H_1(\Sigma_{h,k},\bC)$ and the second largest eigenvalue of $\psi$ on $H_1(\Sigma_{h,k},\bC)$.

\begin{cor}\label{c:gap}
For each $\eps>0$ there exists a finite, abelian cover $\Sigma_{h,k}\to\Sigma_{g,n}$ to which $\psi$ lifts and for which $\gamma<\eps$.
\end{cor}

In other words, as we consider larger and larger abelian covers of the base, new homological eigenvalues appear which get as close to the dilatation as we like.

\section{Acknowledgements}
The author thanks A. Eskin, B. Farb, C. McMullen and A. Zorich for various helpful discussions.

\section{Train tracks and the Teichm\"uller polynomial}
Let $\psi\in\Mod_{g,n}$ be a pseudo-Anosov homeomorphism which fixes a nontrivial integral cohomology class.  As is standard, the suspended $3$--manifold $M_{\psi}$ is hyperbolic and admits infinitely many non--equivalent fibrations which are parametrized by the Thurston unit norm ball (see \cite{FLP}, \cite{ctm2} and \cite{T}).  The Teichm\"uller polynomial was developed in \cite{ctm3} by McMullen in order to study fibered faces of the Thurston unit norm ball.  In this section we will recall some of the background on the Teichm\"uller polynomial.

Recall that a {\bf train track} is a branched, compact $1$--submanifold of $\Sigma_{g,n}$.  We can think of a train track $\tau$ as an embedded graph such that there is a smooth continuation of each path entering a vertex.  Train tracks are required to have global topological properties.  Indeed, the complement $\Sigma_{g,n}\setminus\tau$ is required to be a union of (cusped) subsurfaces of $\Sigma_{g,n}$ such that the double of each component has negative Euler characteristic.  In particular, smooth annuli, monogons and bigons are disallowed as complementary regions.  A {\bf transverse measure} on a train track is an assignment of a nonnegative weight to each edge of $\tau$ in such a way that the sum of the weights entering a vertex is equal to the sum of the weights exiting the vertex.  See \cite{PH} for general background on train tracks.

Train tracks are finite combinatorial objects used to study measured foliations on surfaces.  To each measured foliation, one can associate a measured train track, and conversely one can recover a measured foliation from a measured train track.  Both of these associations are up to some standard notion of equivalence which we will not discuss here but can be found in \cite{PH}.

A pseudo-Anosov homeomorphism $\psi$ has two associated measured foliations, and as such can be canonically assigned two measured train tracks, $\tau$ and $\tau^*$, corresponding to the expanding and contracting foliation respectively.  We will restrict our attention to $\tau$.  It is standard to call $\tau$ an {\bf invariant train track}, since there is an isotopy of $\Sigma_{g,n}$ which sends $\psi(\tau)$ onto $\tau$ in a way which sends vertices to vertices.  We will fix such an isotopy, thus obtaining a well-defined map $P_E$ from the free abelian group on the edges to itself.  The map $P_E$ is defined by taking an edge $e$ of $\tau$ and recording which edges of $\psi(\tau)$ hit $e$, with multiplicity.  It is a standard result that $P_E$ is a Perron--Frobenius matrix, which is to say that for each sufficiently large $n$, each entry of $P_E^n$ is a positive integer.

It is a standard fact that the measure on $\tau$ is positive on each edge and that $\tau$ {\bf fills} $\Sigma_{g,n}$, in the sense that each complementary region of $\tau$ is topologically a (possibly punctured) disk (see \cite{FLP} and \cite{PH}).  This fact gives us the following simple but important observation:

\begin{lemma}
Let $\tau$ be an invariant train track for a pseudo-Anosov homeomorphism.  Then the edges of $\tau$ form the $1$--cells of a CW structure on $\Sigma_{g,n}$.
\end{lemma}
\begin{proof}
This is a general fact about filling train tracks.  If a complementary region is not punctured, it can be used as a $2$--cell.  If a complementary region is punctured, it can be deformation retracted onto its boundary.
\end{proof}

Let $M=M_{\psi}$ be the suspended $3$--manifold of $\Sigma_{g,n}$.  Since $\psi$ has invariant homology, the rank of $H_1(M,\bZ)$ is at least two.  We obtain an infinite abelian cover $S$ of $\Sigma_{g,n}$ by restricting the map \[\pi_1(M)\to H_1(M,\bQ)\] to $\pi_1(\Sigma_{g,n})$.  Write $H$ for the corresponding deck group.  We may lift both $\tau$ and $\psi$ to $S$.  Notice that $\psi$ acts trivially on $H$, so that $\psi$ commutes with the deck group and we can lift the carrying map $\psi(\tau)\to\tau$ in a $\psi$--equivariant way.

The edges and vertices of the total lift of $\tau$ are just edges and vertices of $\tau$ labelled by elements of $H$.  The total lift of $\tau$ still fills $S$.  The carrying map gives us a well--defined map $P_E$ of the free $\bZ[H]$--module on the edges, and it furnishes a matrix with coefficients in $\bZ[H]$.  We can construct a similar map $P_V$ for the vertices, and the ratio of the two characteristic polynomials is the {\bf Teichm\"uller polynomial} of the fibered face determined by $\psi$.  We denote the Teichm\"uller polynomial by $\theta$.  If $\phi:H\to\bC$ is any homomorphism, we obtain a polynomial in one variable with complex coefficients, which we call a {\bf specialization} of $\theta$ and write as either $\theta_{\phi}$ or as $\theta(u,\phi(t))$.

We now recall some facts about the Teichm\"uller polynomial which can all be found in \cite{ctm3}.  The Teichm\"uller polynomial is usually specialized at real cohomology classes (i.e. homomorphisms $\phi:H\to\bR$) in the fibered face.  This way, we obtain real polynomials.  Each rational point in the fibered face corresponds to another fibration of $M$.  The largest root of the Teichm\"uller polynomial at such a rational point (suitably rescaled) is the dilatation of the monodromy.  The function which assigns to each point on the fibered face the largest root of the specialization of $\theta$ is a concave analytic function which tends to infinity on the boundary of the fibered face.

Note that if $\phi$ is an integral cohomology class which arises from a homomorphism $H\to\bZ$, we get an infinite cyclic cover of $\Sigma_{g,n}$ and an associated specialization $\theta_{\phi}$ of the Teichm\"uller polynomial, which is now a polynomial over the ring of Laurent polynomials in one variable.  Since the dilatation blows up at the boundary of the fibered face, it follows that there is no unit $u$ such that the coefficients of $u\cdot \theta_{\phi}$ are all constant Laurent polynomials.  This observation has the following immediate consequence:

\begin{lemma}\label{l:spread1}
Choose any basis $\{t_1,\ldots,t_h\}$ for $H$ and let $u$ be any unit in $\bZ[H]$.  Then for each $i$, the coefficients of $u\cdot \theta$ are not all independent of $t_i$.
\end{lemma}

We finally make a few observations about homological pseudo-Anosov homeomorphisms.  If $\psi$ is homological then $\tau$ admits an orientation.  By this, we mean that we can orient each edge of $\tau$ in such a way that every smooth path in $\tau$ has a coherent orientation.  Furthermore, the carrying map $\psi(\tau)\to\tau$ is orientation preserving.  It follows that $P_E$ not only records the edges which hit a particular edge $e$ of $\tau$ with multiplicity, but with signed multiplicity.  The oriented edges of $\tau$ now provide a basis for $C_1(\Sigma_{g,n},\bZ)$, the one--chains with respect to some cell decomposition of $\Sigma_{g,n}$, and the adjoint of $P_E$ is the map $\psi_*$ induced on $C_1(\Sigma_{g,n},\bZ)$ by $\psi$.  On the cover $S$ of $\Sigma_{g,n}$, the orientation of $\tau$ lifts, as does $\psi$, and $P_E$ is replaced by the corresponding matrix of Laurent polynomials.  This is the fundamental connection between train track, pseudo-Anosov homeomorphisms, and homology.

Now suppose that $\Sigma_{h,k}\to\Sigma_{g,n}$ is a finite abelian cover with deck group $A$, and suppose $\psi$ lifts to this cover.  The $1$--chains $C_1(\Sigma_{h,k},\bC)$ form a module with commuting actions of $\psi$ and $A$.  The $A$--action decomposes $C_1(\Sigma_{h,k},\bC)$ into eigenspaces, and $\psi$ preserves these eigenspaces.  The action of $\psi$ on this eigenspace is described by $P_E$, except that the Laurent polynomial indeterminates now act by roots of unity according to particular representation of $A$.

Since we can build finite abelian covers by writing down homomorphisms from $H$ to various finite groups, it becomes clear that the action of $\psi$ on train tracks on various finite abelian covers of $\Sigma_{g,h}$ is determined by finite image homomorphisms $\phi:H\to S^1$.  The action of $\psi$ on one--chains is obtained by specializing $P_E$ at $\phi$.

Write $p_e$ for the characteristic polynomial of $P_E$.  Notice that taking the largest root of $p_e$ yields a continuous function of the coefficients of $p_e$.  We thus obtain a continuous function $\rho$ on $\Hom(H,S^1)\cong\bU^h$ which assigns to a point $z$ the modulus of the largest eigenvalue of the specialization $\theta_z$ of the Teichm\"uller polynomial.  Note that we lose no information by specializing $\theta$ as opposed to $p_e$, since the roots of the characteristic polynomial of $P_V$ are all roots of unity and morally we are only interested in eigenvalues which lie off the unit circle.

\section{Twisted homology, Alexander polynomials, representations and covers}
Let $X$ be a CW complex with nontrivial fundamental group $\pi$.  Recall that one can consider twisted homology of $X$.  Indeed, let $\rho:\pi\to GL(V)$ be a representation of $\pi$.  One can consider the twisted homology of $X$ with coefficients in $V$, written $H_*(X,V_{\rho})$.  Twisted homology is related to covers in an essential way.  Let $\widetilde{X}$ be the cover of $X$ corresponding to the kernel of $\rho$, equipped with a lifted CW structure.  One first considers the complex $C_*(\widetilde{X},\bZ)$.  This complex is equipped with a natural action of the group ring $\bZ[\pi]$.  We define the complex \[C_*(X,V_{\rho})=C_*(\widetilde{X},\bZ)\otimes_{\bZ[\pi]} V.\]  The twisted homology $H_*(X,V_{\rho})$ is just the homology of this complex.

We are primarily interested in $H_1(\Sigma_{g,n},V_{\chi})$, where $\Sigma_{g,n}$ is a surface and $\chi$ is a finite, one--dimensional representation of $\pi_1(\Sigma_{g,n})$.  Any such representation factors through a map \[\chi:H_1(\Sigma_{g,n},\bZ)\to\bC^*\] whose image is finite.  Write $\Sigma_{h,k}$ for the covering space corresponding to the kernel of $\chi$.  The twisted homology group is then $H_1(\Sigma_{g,n},\bC_{\chi})$ and can be identified with the subspace of $H_1(\Sigma_{h,k},\bC)$ where the action of $\pi_1(\Sigma_{g,n})$ is given by $\chi$, i.e. the $\chi$--eigenspace of $H_1(\Sigma_{h,k},\bC)$.

\subsection{The Teichm\"uller polynomial and the Alexander polynomial}
We now make some further remarks about the relationship between train tracks and homology via the Alexander polynomial (see \cite{ctm4} for more details).  The Alexander polynomial of a $3$--manifold $M$ is an element of $\bZ[H_1(M,\bZ)/(torsion)]$.  One way to define the Alexander polynomial $A$ is to consider the cover $M'\to M$ corresponding to $H_1(M,\bZ)/(torsion)$.  Let $p$ be a basepoint of $M$ and $p'\subset M'$ its preimage.  The {\bf Alexander module} of $M$ is the $\bZ[H_1(M,\bZ)/(torsion)]$--module $H_1(M',p',\bZ)$.  The {\bf Alexander ideal} is the first elementary ideal of the Alexander module, and the {\bf Alexander polynomial} is the greatest common divisor of the elements in the elements in the Alexander ideal.

Returning to the situation at hand, after choosing a fibration $M\to S_1$, we can write \[H_1(M,\bZ)/(torsion)\cong\bZ\oplus H.\]  Thus, the Alexander polynomial can be written as a polynomial $A(u,t)$, where $t\in H$.  The relationship between cohomology of finite abelian covers and Alexander polynomials is given by the following result in \cite{ctm4}:

\begin{prop}
An Alexander polynomial in more than one variable defines the maximal hypersurface in the character variety such that $\dim H^1(M,\bC_{\chi})>0$ whenever $A(\chi)=0$.
\end{prop}

The relationship between Alexander polynomials and the action of $\psi$ on the homology groups of finite covers is the following observation, which is well--known and appears in \cite{ctm3}, for instance:
\begin{prop}\label{p:1}
Let $\chi:H\to S^1$ be a character of $H$.  Then the polynomial $A(u,\phi(t))$ is the characteristic polynomial of the action of $\psi$ on $H_1(\Sigma,\bC_{\chi})$.
\end{prop}

The relationship between the Teichm\"uller polynomial for a fibration with monodromy $\psi$ and an orientable invariant foliation is given as follows (see \cite{ctm3} for a proof):

\begin{prop}\label{p:2}
The Alexander polynomial $A(u,t)$ divides the Teichm\"uller polynomial $\theta(u,t)$.
\end{prop}

On the other hand, the Teichm\"uller polynomial gives more information than just the action of $\psi$ on the homology of certain finite abelian covers of $\Sigma_{g,n}$.  Indeed, let $\psi$ act on the homology $H_1(\Sigma_{g,n},\bZ)$ with trace smaller than zero.  Then there is no basis for $H_1(\Sigma_{g,n},\bZ)$ with respect to which the action of $\psi$ is Perron--Frobenius.

\section{Perron--Frobenius matrices over Laurent polynomial rings}
In this short section, we prove a result about matrices with entries in Laurent polynomial rings.  If $p\in\bQ[t,t^{-1}]$, define the {\bf spread} of $p$ to be the absolute value of the difference between the highest and lowest exponent of $t$ occurring in $p$.
\begin{lemma}\label{l:spread}
Suppose $\{p_1,\ldots,p_k\}$ and $\{q_1,\ldots,q_k\}$ are two collections of nonzero rational Laurent polynomials with positive coefficients.  Then the spread of \[\sum_{i=1}^k p_iq_i\] is at least as large as the maximum of the spreads of the $\{p_i\}$ and the $\{q_i\}$.
\end{lemma}
\begin{proof}
Since each $p_i$ and $q_i$ is nonzero and has only positive coefficients, it is clear that no cancellation can occur, so that the spread can only increase.
\end{proof}

\begin{lemma}
Let $M\in M_n(\bQ[t,t^{-1}])$ be Perron--Frobenius, and suppose that there is an entry of $M$ whose spread is at least one.  Then there is a $k$ such that each entry of $M^k$ has spread at least one.
\end{lemma}
\begin{proof}
Passing to a power of $M$ if necessary, each entry of $M$ is a nonzero Laurent polynomial with positive coefficients.  Without loss of generality, $a_{1,1}$ is not a unit.  In particular, the spread of $a_{1,1}$ is nonzero.  By Lemma \ref{l:spread}, the spread of each entry in the first row of $M^2$ is least one.  By the same argument, the spread of each entry of $M^3$ is at least one.
\end{proof}

In the language of this section, we have a convenient rephrasing of Lemma \ref{l:spread1}:

\begin{lemma}\label{l:spread2}
Choose a basis $\{t_1,\ldots,t_h\}$ for $H$.  There is a $k$ such that each entry of $P_E^k$ has spread at least $1$ in each of the variables $\{t_1,\ldots,t_h\}$.
\end{lemma}

\section{The point--wise spectral gap}
We are now ready to prove the results leading up to Theorem \ref{t:main}.

\begin{proof}[Proof of Proposition \ref{p:pf}]
The fundamental observation for the proof of the proposition is the following: if $\{z_1,\ldots,z_n\}$ are complex numbers of modulus one and $\{a_1,\ldots,a_n\}$ are positive integers then the modulus of \[\sum_ia_iz_i\] is maximized when $\{z_1,\ldots,z_n\}$ all have the same argument.

Since $M$ is Perron--Frobenius, we may assume that each entry $m_{i,j}$ of $M$ is a nonconstant Laurent polynomial in one variable with spread at least one.  Write $K>1$ for the spectral radius of $\phi_0(M)$.  Consider the matrices $\phi(M)^n=(a_{i,j,n})$ and $\phi_0(M)^n=(b_{i,j,n})$.  Since $\phi(t)$ is an irrational point on the circle, it follows that \[C=\max_{i,j}\frac{|a_{i,j,1}|}{b_{i,j,1}}<1.\]

Consider the entries of $\phi(M)^{n+1}$.  We have the following inequalities: \[|a_{i,j,n+1}|=\Big|\sum_{k=1}^ma_{i,k,n}a_{k,j,1}\Big|\leq\sum_{k=1}^m|a_{i,k,n}||a_{k,j,1}|<\sum_{k=1}^m|a_{i,k,n}||b_{k,j,1}|,\] where $m$ is the dimension of $M$.  Notice that the rightmost expression exceeds the second rightmost by a factor of at least $C$.  It follows that the sum of the moduli of the entries of $\phi(M)^n$ is exceeded by the sum of the entries of $\phi_0(M)^n$ by a factor of at least $C^n$.  It follows that in the $\ell_1$--norm, we have \[||\phi(M)^n||_1\leq C^n\cdot ||\phi_0(M)^n||.\]
It follows that the spectral radius of $\phi(M)$ is at most $C\cdot K$.
\end{proof}

\begin{proof}[Proof of Corollary \ref{c:teich}]
Choose a basis $\{t_1,\ldots,t_h\}$ for $H$.  Since $\phi$ is a non--torsion point of the representation variety $\Hom(H,S^1)$, there is at least one basis element for $H$, say $t_1$, which is sent to an irrational point on $S^1$.  By Lemma \ref{l:spread2}, we may assume that each entry of $P_E$ has spread at least one in $t_1$, possibly after passing to a power.  By Proposition \ref{p:pf}, we have that the spectral radius of $\phi_0(P_E)$ is strictly larger than that of $\phi(P_E)$.
\end{proof}

\begin{proof}[Proof of Theorem \ref{t:main}]
For each nonidentity $\phi\in\Hom(H,S^1)$, we need to show that the spectral radius of the action of $\psi$ on $H_1(\Sigma_{g,n},\bC_{\phi})$ is strictly smaller than the dilatation of $\psi$.  The conclusion of the theorem will follow from the continuity of the function $\rho$ and the compactness of $\Hom(H,S^1)$.

For nonidentity torsion points in $\Hom(H,S^1)$, this follows from the fact that torsion points of $\Hom(H,S^1)$ give rise to finite abelian covers of $\Sigma_{g,n}$ to which $\psi$ lifts.  The dilatation of $\psi$ is a simple eigenvalue of the action of $\psi$ on the homology of that cover, whence it follows that the spectral radius on $H_1(\Sigma_{g,n},\bC_{\phi})$ is strictly smaller than the dilatation.

For a non--torsion point $\phi$, we have that the largest root of $\theta(u,\phi(t))$ is strictly smaller than the dilatation of $\psi$, by Corollary \ref{c:teich}.  Since the Alexander polynomial $A(u,\phi(t))$ is the characteristic polynomial of the action of $\psi$ on $H_1(\Sigma_{g,n},\bC_{\phi})$ and since $A(u,\phi(t))$ divides $\theta(u,\phi(t))$, we have that the spectral radius is in fact smaller than the dilatation.
\end{proof}

\begin{proof}[Proof of Corollary \ref{c:gap}]
Since the Alexander polynomial varies continuously on $\Hom(H,S^1)$ as does its largest root, and since its largest root at $\phi_0$ is the dilatation of $\psi$, the result follows.
\end{proof}

\section{Comparison with other homological representations}
When $g=0$, mapping class groups of the surfaces $\{\Sigma_{g,n}\}$ are essentially just braid groups.  We will assume in this case that the surface $\Sigma_{g,n}$ is just the $n$--times punctured disk $D_n$, so that the mapping class group is identified with the braid group $B_n$.  The braid group $B_n$ has the pure braid group $P_n$ as a subgroup of finite index, and it arises as the kernel of the permutation action of $B_n$ on the punctures of $D_n$.  The groups $B_n$ and $P_n$ have two well--studied homological representations, namely the Burau representation and the Gassner representation (see \cite{Bir} for more details).

The Burau representation is given by considering the infinite cyclic cover of $D_n$ obtained by sending a small loop about each puncture of $D_n$ to a fixed generator of $\bZ$.  The Burau representation is the associated homology representation.  The Gassner representation is the analogous representation of $P_n$, except that we consider the universal abelian cover of $D_n$.

These representations take a braid $\beta$ and return a matrix over a Laurent polynomial ring.  It can be shown that specializations of these matrices, given by sending the indeterminates to roots of unity, describe the actions of braid (resp. pure braid groups) on the homologies of all finite abelian covers with equal branching over all the punctures (resp. all finite abelian covers of the disk).  For more details, consult \cite{KobGeom}.

The Burau and the Gassner representations are useful in that one can easily associate a characteristic polynomial to each braid, and the roots of the specializations give eigenvalues of the actions of the braid on various deck group eigenspaces of the homology of covers.

\subsection{The simplest pseudo-Anosov braid}
Consider the group $B_3$, which is generated by the two braids $\sigma_1$ and $\sigma_2$, which interchange the first two and second two strands respectively, in the same direction.  It is well--known that the braid $\beta=\sigma_1\sigma_2^{-1}$ is pseudo-Anosov.  In fact, there is a double cover of $D_3$ which is homeomorphic to a torus with three punctures and one boundary component, and a lift of $\beta$ acts by the matrix \[\begin{pmatrix}2&1\\1&1\end{pmatrix}.\]  Under the Burau representation, we have that \[\sigma_1\mapsto\begin{pmatrix}t&t\\0&1\end{pmatrix},\] and \[\sigma_2\mapsto\begin{pmatrix}1&0\\t^{-1}&t^{-1}\end{pmatrix}.\]

Computing the characteristic polynomial, we obtain $p(u,t)=1-u(1+t+t^{-1})+u^2$, up to a unit in $\bZ[t,t^{-1}]$.  We can write a program which computes the largest root of $p(u,\zeta)$ as $\zeta$ ranges over roots of unity.  We can see that there is a unique maximum of $(3+\sqrt{5})/2$ at $\zeta=1$.  This follows from the fact that the function $1+t+t^{-1}$ on the unit circle is just the function $1+2\Re(t)$, which achieves a unique maximum on the circle.

\end{document}